\newcommand{\dataversione}{\today}

\documentclass[11pt,reqno,a4paper]{amsart}

\usepackage{mathrsfs, graphics}
\usepackage{amssymb}
\usepackage[notref,notcite,final]{showkeys}
\usepackage{hyperref}\hypersetup{colorlinks=true, citecolor=blue}
\usepackage{graphicx, epstopdf}\graphicspath{{./immaginiprova/}}

\numberwithin{equation}{section}

\newtheoremstyle{mytheorem}
{}
{}
{\it}
{\parindent}
{\bf}
{.}
{ }
{\thmnumber{#2.~}\thmname{#1}\thmnote{~\rm#3}}

\newtheoremstyle{myremark}
{}
{}
{\rm}
{\parindent}
{\bf}
{.}
{ }
{\thmnumber{#2.~}\thmname{#1}\thmnote{~\rm#3}}

\newtheoremstyle{myparagraph}
{}
{}
{\rm}
{\parindent}
{\bf}
{.}
{ }
{\thmnumber{#2.~}\thmname{#1}\thmnote{#3}}

\theoremstyle{mytheorem}

\newtheorem{theorem}[subsection]{Theorem}
\newtheorem{lemma}[subsection]{Lemma}

\newtheorem{proposition}[subsection]{Proposition}

\theoremstyle{myremark}
\newtheorem{remark}[subsection]{Remark}

\theoremstyle{myparagraph}

\newtheorem*{parag*}{}

\makeatletter
\def\@secnumfont{\sc}
\def\section{\@startsection{section}{1}%
\z@{1.5\linespacing\@plus .2\linespacing}{.7\linespacing}%
{\normalfont\sc\centering}}
\makeatother

\makeatletter
\def\ps@headings{\ps@empty 
 \def\@evenhead{%
  \setTrue{runhead}%
  \normalfont\footnotesize 
  \rlap{\thepage}\hfil 
  \def\thanks{\protect\thanks@warning}%
  \leftmark{}{}\hfil}%
 \def\@oddhead{%
  \setTrue{runhead}%
  \normalfont\footnotesize\hfil 
  \def\thanks{\protect\thanks@warning}%
  \rightmark{}{}\hfil \llap{\thepage}}%
\let\@mkboth\markboth}
\makeatother
\makeatletter
\renewenvironment{proof}[1][\proofname]{\par 
  \pushQED{\qed}%
  \normalfont \topsep6\p@\@plus6\p@\relax 
  \trivlist 
  \itemindent\normalparindent 
  \item[\hskip\labelsep 
    \bfseries 
    #1\@addpunct{.}]\ignorespaces 
}{%
  \popQED\endtrivlist\@endpefalse 
} 
\providecommand{\proofname}{Proof}
\makeatother
\newcommand{\R}{\mathbb{R}}
\newcommand{\N}{\mathbb{N}}

\newcommand{\dV}{d_V\kern-1pt}

\makeatletter
\@addtoreset{footnote}{section}
\makeatother

\begin{document}

	%
\pagestyle{empty}
\pagestyle{myheadings}
\markboth%
{\underline{\centerline{\hfill\footnotesize%
\textsc{Andrea Marchese}%
\vphantom{,}\hfill}}}%
{\underline{\centerline{\hfill\footnotesize%
\textsc{BV homeomorphisms}%
\vphantom{,}\hfill}}}

	%
\thispagestyle{empty}

~\vskip -1.1 cm

	%
{\footnotesize\noindent
[version:~\dataversione]%
\hfill
}

\vspace{1.7 cm}

	%
{\large\bf\centering
Residually many BV homeomorphisms\\
map a null set onto a set of full measure\\
}

\vspace{.6 cm}

	%
\centerline{\sc Andrea Marchese}

\vspace{.8 cm}

{\rightskip 1 cm
\leftskip 1 cm
\parindent 0 pt
\footnotesize

	%
{\sc Abstract.}
Let $Q$ be the open unit square in $\mathbb{R}^2$. We prove that in a natural complete metric space of $BV$ homeomorphisms $f:Q\rightarrow Q$ with $f_{|\partial Q}=Id$, residually many homeomorphisms 
(in the sense of Baire categories) map a null set onto a set of full measure, and vice versa. Moreover we observe that, for $1\leq p<2$, the family of $W^{1,p}$ homemomorphisms satisfying the above property is of first category.
\par
}

%
%

\section{Introduction}
\label{s1}
\subsection{Notation and main result}
Denote by $|\cdot|_\infty$ the norm on $\R^4$ given by 
$$|(a,b,c,d)|_\infty=\max\{|a|,|b|,|c|,|d|\}.$$

Let $Q=(0,1)^2$ be the open unit square in $\R^2$.  
Consider a map $f:Q\rightarrow Q$ and denote by $f_1,f_2:\R^2\to\R$ its components, relative to the usual coordinates in $\R^2$.

Denote the \emph{variation} of $f$ in $Q$ by 
$$Var(f,Q):=\sup\left\{\int_Q(f_1{\rm{div}}\phi_1+f_2{\rm{div}}\phi_2)dx: |\phi(x)|_\infty\leq 1\;{\rm{for\:all}}\; x\in Q\right\},$$
where $\phi=(\phi_1,\phi_2)\in C^1_c(Q,\R^2\times\R^2)$ and the integration is with respect to the Lebesgue measure. Compare this definition with \cite[Definition 3.4]{AFP}: we do not use the usual notation $V(f,Q)$, to remark that we do not compute the norm of $\phi$ in the standard way. We say that a map $f:Q\to Q$ is of \emph{bounded variation} (shortly: $BV$) if $Var(f,Q)$ is finite.

Fix a constant $M>2$ and consider the set
$$X:=\{f:Q\rightarrow Q:\; f\;{\rm{is\;a\;}}BV\;{\rm{homeomorphism}},\; f_{|\partial Q}=Id,\; Var(f,Q)<M\},$$
and the distance on $X$
$$d_X(f,g):=\|f-g\|_\infty+\|f^{-1}-g^{-1}\|_\infty+\left|\frac{1}{M-Var(f,Q)}-\frac{1}{M-Var(g,Q)}\right|.$$
We will prove in \S \ref{s1} that $(X,d_X)$ is a complete metric space.
Now let us consider the following subset of $X$:
$$A:=\{f\in X: \exists E\subset Q, E \mbox{ Borel, } |E|=0, |f(E)|=1\},$$
where we denoted by $|E|$ the Lebesgue measure of $E$. 
The main result of the present paper is the following
\begin{theorem}\label{main}
The set $A$ is residual in $X$, i.e. it contains the intersection of countably many open dense subsets of $X$.
\end{theorem}
The Baire theorem (see for instance \cite{Ru}) implies that the set $A$ is non-empty and, more precisely, that it is dense in $X$.

\subsection{Strategy for the proof of Theorem \ref{main}}
We introduce the following family of subsets of $X$. For every $n\in\N$, we denote
$$A_n:=\{f\in X: \exists E\subset Q, |E|<1/n, |f(E)|>1-1/n\},$$
where the set $E$ is a union of finitely many pairwise disjoint open triangles (the number of such triangles may depend both on $n$ and on the function $f$).

It is easy to see that the set $A$ contains the intersection of the $A_n$'s (see \S\ref{s4}). Hence, to prove Theorem \ref{main} it suffices to show that the sets $A_n$ are open and dense in $X$. The openness is an easy issue (see Lemma \ref{lopen}), while density is more delicate (see Lemma \ref{ldense}), so we will give here a sketch of the proof of the second property. We prove both properties in \S \ref{s4}.

Fix $n\in \N$, $f\in X$ and $\varepsilon>0$. We want to find $f_{\varepsilon}\in A_n$ with $d_X(f,f_{\varepsilon})<\varepsilon$. Firstly we use the main result of \cite{PR} to obtain an orientation preserving, (finitely) piecewise affine homeomorphism $g_{\varepsilon}\in X$ with
$$d_X(f,g_{\varepsilon})<\varepsilon/4.$$

Then we take a finite triangulation of $Q$ such that $g_{\varepsilon}$ is affine on each triangle. If necessary, we can refine such triangulation in order to obtain a new finite triangulation $\tau$ such that 
the diameter of all triangles $T\in\tau$ and of their images through $g_{\varepsilon}$ do not exceed $\varepsilon/8$. This ensures that any perturbation $h$ of $g_\varepsilon$,
which agrees with $g_\varepsilon$ on $\partial T$ for every $T\in\tau$, satisfies 
$${\|h-g_{\varepsilon}\|}_\infty+{\|h^{-1}-g^{-1}_{\varepsilon}\|}_\infty\leq\varepsilon/4.$$

Finally we modify the homeomorphism $g_{\varepsilon}$ inside each triangle $T\in\tau$ in order to obtain a new orientation preserving homeomorphism $f_{\varepsilon}\in X$ with the following properties:
\begin{enumerate}
\item $f_{\varepsilon}$ is (finitely) piecewise affine on each triangle $T\in\tau$;
\item for every $T\in\tau$ it holds ${f_{\varepsilon}}_{|\partial T}={g_{\varepsilon}}_{|\partial T};$
\item $\left|\frac{1}{M-Var(f_{\varepsilon},Q)}-\frac{1}{M-Var(g_{\varepsilon},Q)}\right|\leq\varepsilon/2$;
\item for every $T\in\tau$ there exists a set $F\subset T$ which is the union of finitely many pairwise disjoint open triangles and satisfies
$$\frac{Area(F)}{Area(T)}<1/n;\;\frac{Area(f_{\varepsilon}(F))}{Area(f_{\varepsilon}(T))}>1-1/n.$$
\end{enumerate}
Clearly the property of the triangulation $\tau$ implies that 
$$\|f_{\varepsilon}-g_{\varepsilon}\|_\infty+\|f_{\varepsilon}^{-1}-g_{\varepsilon}^{-1}\|_\infty<\varepsilon/4$$
and, together with property (3), this implies that $d_X(f,f_{\varepsilon})<\varepsilon$. Moreover properties (1),(2) and (4) imply that $f_\varepsilon\in A_n$.

The construction of $f_{\varepsilon}$ starting from $g_{\varepsilon}$ uses a piecewise affine homeomorphism $\phi_n$ (defined in \S \ref{s3}), which maps a square $Q$ onto a parallelogram $P$ and coincides with the affinity between $Q$ and $P$ on $\partial Q$. This map is similar in spirit to the ``basic building block'' used in \cite{He} to construct Sobolev homeomorphisms with zero distributional Jacobian almost everywhere. The 
main difference between the two maps is that, although in both cases the aim is to map a small subset $F$ of $Q$ onto a (proportionally) much larger set $F'$, with a small cost in the variation, we want in addition that $F'$ is almost a set of full measure in $P$. 

\begin{parag*}[Acknowledgements]
The author would like to warmly thank Guido De Philippis and Aldo Pratelli for valuable suggestions and helpful comments. The author would like to acknowledge the support of the Max Planck Institute for Mathematics in Leipzig, where he was hosted when this note was written.
\end{parag*}

\section{The metric space $X$}
\label{s2}
In this section we prove that the pair $(X,d_X)$ defined in the Introduction actually identifies a complete metric space. Since there are no doubts that $d_X$ defines a metric on $X$, we will focus on the completeness.
\begin{proposition}
The metric space $(X,d_X)$ is complete.
\end{proposition}
\begin{proof}
It is a well known result in functional analysis that a sequence $(f_i)_{i\in\N}$ of $BV$ maps which is Cauchy with respect to the supremum norm and such that the variations $Var(f_i,Q)<M$ are equi-bounded 
converges to a $BV$ map $f$ with $Var(f,Q)\leq M$ (See e.g. propositions 3.6 and 3.13 of \cite{AFP}: clearly our re-norming of $\R^4$ does not affect the validity of such statements). We need to prove that if the sequence is Cauchy with respect to the distance $d_X$, then the limit $f$ remains a homeomorphism, and moreover 
$Var(f,Q)<M$. To prove that $f$ is a homeomorphism it is sufficient to observe that the sequence of continuous functions $(f_i^{-1})_{i\in\N}$ is Cauchy with respect to the supremum norm and therefore it 
converges to a continuous function $g$, which is the inverse of $f$ (the latter is a trivial fact of general topology). 

Assume now by contradiction that $Var(f,Q)=M$. The lower semicontinuity of the variation with respect to the uniform convergence implies that 
$$\lim_{i\to\infty}Var(f_i,Q)=M,$$ 
which implies that, for every fixed $m\in\N$ the quantity
$$\left|\frac{1}{M-Var(f_m,Q)}-\frac{1}{M-Var(f_j,Q)}\right|$$
is unbounded in $j$, hence $(f_i)_{i\in\N}$ is not Cauchy with respect to $d_X$, which is a contradiction.
\end{proof}

\begin{remark}
The completeness of the metric space $(X,d_X)$ is necessary in order to apply the Baire theorem. A non-strict inequality on the variation in the definition of $X$ would be arguably a more natural choice, and in particular this would allow us to drop the last term in the definition of $d_X$. Nevertheless we introduced such metric space, because in order both to perform the piecewise affine approximation of \cite{PR} and to modify such approximation using the 
homeomorphism $\phi_n$, 
we may need to increase the variation of a small amount. 

It is actually possible to circumvent this issue even in the setting mentioned above, i.e. when we just require $Var(f,Q)\leq M$: indeed it is sufficient to approximate preliminarily a BV homeomorphism $f$ by homeomorphisms having smaller variation. This can be always achieved if $Var(f,Q)>2$, by ``interpolating'' a contraction of the original homeomorphism $f$ in an inner square
concentric to $Q$ and the identity on the outer frame.
\end{remark}

\section{The homeomorphism $\phi_n$}
\label{s3}
In this section we prove the following 

\begin{proposition}\label{propphin}
For every affine homeomeorphism $\phi$ defined on a square $Q$ with edges parallel to the coordinate lines and for every $n\in\N$ there exists piecewise affine map $\phi_n$ on Q and a set $F_n$ which is a finite union of pairwise disjoint open triangles such that 
\begin{enumerate}
\item $|Var(\phi_n,Q)-Var(\phi,Q)|\leq (1/n) Var(\phi,Q)$;
\item ${\phi_n}_{|\partial Q}={\phi}_{|\partial Q}$;
\item $\frac{Area(F_n)}{Area(Q)}<1/n;\;\frac{Area(\phi_n(F_n))}{Area(\phi(Q))}>1-1/n.$
\end{enumerate}
\end{proposition}
\begin{proof}
Let $(x,y)$ denote the usual coordinates on the plane and let $\overline Q$ be the unit square $[0,1]\times[0,1]$. Consider a linear map 
$$A:=\left(\begin{array}{cc}
a & b\\
c & d\end{array}\right)$$
with $\det(A)>0$. Clearly $A$ identifies the linear, orientation preserving homeomorphism $\psi$ which maps the points $(1,0)$ and $(0,1)$ in $(a,c)$ and $(b,d)$ respectively.
Fix $n\in\N$, $n>2$. We will define a piecewise affine, orientation preserving homeomorphism $\phi_n$ such that 

\begin{equation}\label{eqF2}
Var(\phi_n,Q)\leq(1+(2/n^{1/2}))Var(\psi,Q),\quad\mbox{ and \quad${\phi_n}_{|\partial Q}={\psi}_{|\partial Q}$.}
\end{equation} 
Moreover we construct $\phi_n$ in such a way that there exists a set $F\subset Q$ which is a union of finitely many pairwise disjoint open triangles, satisfying 
\begin{equation}\label{eqF}
Area(F)<\frac{1}{n^{1/2}};\quad Area(\phi_n(F))\geq\left(1-\frac{1}{2n}\right)\left(1-\frac{1}{n^{1/2}}\right)\det(A). 
\end{equation}
Since $n$ is arbitrary, \eqref{eqF2} and \eqref{eqF} would imply the validity of conditions $(1)-(3)$ of the Proposition for the choice $\phi=\psi$. The validity for general $\phi$ is obtained simply by translating and rescaling.\\

For $i=0,\ldots,n^2-1$, let $R_i$ be the rectangle 
$$R_i:=[0,1]\times[i/n^2,(i+1)/n^2].$$
Denote 
$$R':=[1/n,1-(1/n)]\times[0,1/n^{5/2}]\subset R_0$$
and
$$R'':=[1/n,1-(1/n)]\times[1/n^{5/2},1/n^2]\subset R_0.$$
Finally consider $R_0\setminus(R'\cup R'')$. We define a partition of the left rectangle 
$$R''':=[0,1/n]\times[0,1/n^2]$$ 
and on the right rectangle we define the symmetric partition with respect to the axis $x=1/2$.
Let us write 
$$R'''=T_1\cup T_2\cup T_3\cup T_4,$$
where (see Figure \ref{F1}):
\begin{itemize}
\item $T_1$ has vertices in $(0,0)$, $(1/n,0)$, and $(0,1/n^{5/2})$,
\item $T_2$ has vertices in $(0,1/n^{5/2})$, $(1/n,0)$, and $(1/n,1/n^{5/2})$,
\item $T_3$ has vertices in $(0,1/n^{5/2})$, $(1/n,1/n^{5/2})$, and $(1/n,1/n^2)$,
\item $T_4$ has vertices in $(0,1/n^{5/2})$, $(1/n,1/n^2)$, and $(0,1/n^2)$.
\end{itemize}

\begin{figure}[htbp]
\begin{center}
\scalebox{1}{
\setlength{\unitlength}{4144sp}%
\begingroup\makeatletter\ifx\SetFigFont\undefined%
\gdef\SetFigFont#1#2#3#4#5{%
  \reset@font\fontsize{#1}{#2pt}%
  \fontfamily{#3}\fontseries{#4}\fontshape{#5}%
  \selectfont}%
\fi\endgroup%
\begin{picture}(5424,5064)(-11,-5113)
\thinlines
{\color[rgb]{0,0,0}\put(901,-3661){\line( 1, 0){4500}}
}%
{\color[rgb]{0,0,0}\put(4501,-3661){\line( 0,-1){900}}
}%
{\color[rgb]{0,0,0}\put(1801,-3661){\line( 0,-1){900}}
}%
{\color[rgb]{0,0,0}\put(1801,-4336){\line(-1, 0){900}}
}%
{\color[rgb]{0,0,0}\put(4501,-4336){\line( 1, 0){900}}
}%
{\color[rgb]{0,0,0}\put(901,-4336){\line( 4, 3){900}}
}%
{\color[rgb]{0,0,0}\put(4501,-3661){\line( 4,-3){900}}
}%
{\color[rgb]{0,0,0}\put(5401,-4336){\line(-4,-1){900}}
}%
{\color[rgb]{0,0,0}\multiput(901,-2761)(45.00000,0.00000){101}{\makebox(1.5875,11.1125){\tiny.}}
}%
{\color[rgb]{0,0,0}\multiput(901,-1861)(45.00000,0.00000){101}{\makebox(1.5875,11.1125){\tiny.}}
}%
{\color[rgb]{0,0,0}\multiput(901,-961)(45.00000,0.00000){101}{\makebox(1.5875,11.1125){\tiny.}}
}%
{\color[rgb]{0,0,0}\multiput(1801,-3661)(0.00000,45.00000){81}{\makebox(1.5875,11.1125){\tiny.}}
}%
{\color[rgb]{0,0,0}\multiput(4501,-3661)(0.00000,45.00000){81}{\makebox(1.5875,11.1125){\tiny.}}
}%
{\color[rgb]{0,0,0}\multiput(901,-2536)(45.00000,0.00000){101}{\makebox(1.5875,11.1125){\tiny.}}
}%
{\color[rgb]{0,0,0}\multiput(901,-1636)(45.00000,0.00000){101}{\makebox(1.5875,11.1125){\tiny.}}
}%
{\color[rgb]{0,0,0}\multiput(901,-736)(45.00000,0.00000){101}{\makebox(1.5875,11.1125){\tiny.}}
}%
{\color[rgb]{0,0,0}\multiput(901,-3436)(42.85714,-10.71429){22}{\makebox(1.5875,11.1125){\tiny.}}
}%
{\color[rgb]{0,0,0}\multiput(901,-2536)(42.85714,-10.71429){22}{\makebox(1.5875,11.1125){\tiny.}}
}%
{\color[rgb]{0,0,0}\multiput(901,-1636)(42.85714,-10.71429){22}{\makebox(1.5875,11.1125){\tiny.}}
}%
{\color[rgb]{0,0,0}\multiput(901,-736)(42.85714,-10.71429){22}{\makebox(1.5875,11.1125){\tiny.}}
}%
{\color[rgb]{0,0,0}\multiput(4501,-3661)(42.85714,10.71429){22}{\makebox(1.5875,11.1125){\tiny.}}
}%
{\color[rgb]{0,0,0}\multiput(4501,-2761)(42.85714,10.71429){22}{\makebox(1.5875,11.1125){\tiny.}}
}%
{\color[rgb]{0,0,0}\multiput(4501,-1861)(42.85714,10.71429){22}{\makebox(1.5875,11.1125){\tiny.}}
}%
{\color[rgb]{0,0,0}\multiput(4501,-961)(42.85714,10.71429){22}{\makebox(1.5875,11.1125){\tiny.}}
}%
{\color[rgb]{0,0,0}\multiput(901,-3436)(36.00000,27.00000){26}{\makebox(1.5875,11.1125){\tiny.}}
}%
{\color[rgb]{0,0,0}\multiput(901,-2536)(36.00000,27.00000){26}{\makebox(1.5875,11.1125){\tiny.}}
}%
{\color[rgb]{0,0,0}\multiput(901,-1636)(36.00000,27.00000){26}{\makebox(1.5875,11.1125){\tiny.}}
}%
{\color[rgb]{0,0,0}\multiput(901,-736)(36.00000,27.00000){26}{\makebox(1.5875,11.1125){\tiny.}}
}%
{\color[rgb]{0,0,0}\multiput(4501,-61)(36.00000,-27.00000){26}{\makebox(1.5875,11.1125){\tiny.}}
}%
{\color[rgb]{0,0,0}\multiput(4501,-961)(36.00000,-27.00000){26}{\makebox(1.5875,11.1125){\tiny.}}
\multiput(5401,-1636)(-45.00000,0.00000){2}{\makebox(1.5875,11.1125){\tiny.}}
}%
{\color[rgb]{0,0,0}\multiput(4501,-1861)(36.00000,-27.00000){26}{\makebox(1.5875,11.1125){\tiny.}}
}%
{\color[rgb]{0,0,0}\multiput(4501,-2761)(36.00000,-27.00000){26}{\makebox(1.5875,11.1125){\tiny.}}
}%
{\color[rgb]{0,0,0}\put(1801,-4336){\line( 1, 0){2700}}
}%
{\color[rgb]{0,0,0}\put(904,-4325){\line( 4,-1){900}}
}%
{\color[rgb]{0,0,0}\multiput(901,-3436)(45.00000,0.00000){101}{\makebox(1.5875,11.1125){\tiny.}}
}%
{\color[rgb]{0,0,0}\put(901,-61){\line( 0,-1){4500}}
\put(901,-4561){\line( 1, 0){4500}}
\put(5401,-4561){\line( 0, 1){4500}}
\put(5401,-61){\line(-1, 0){4500}}
}%
{\color[rgb]{0,0,0}\put(901,-5011){\line( 1, 0){900}}
}%
{\color[rgb]{0,0,0}\put(901,-4921){\line( 0,-1){180}}
}%
{\color[rgb]{0,0,0}\put(1801,-4921){\line( 0,-1){180}}
}%
{\color[rgb]{0,0,0}\put(  1,-4561){\line( 1, 0){ 90}}
}%
{\color[rgb]{0,0,0}\put(  1,-3661){\line( 1, 0){ 90}}
}%
{\color[rgb]{0,0,0}\put( 46,-3661){\line( 0,-1){900}}
}%
{\color[rgb]{0,0,0}\put(181,-4336){\line( 0,-1){225}}
}%
{\color[rgb]{0,0,0}\put(136,-4561){\line( 1, 0){ 90}}
}%
{\color[rgb]{0,0,0}\put(136,-4336){\line( 1, 0){ 90}}
}%
\put(2971,-4516){\makebox(0,0)[lb]{\smash{{\SetFigFont{12}{14.4}{\rmdefault}{\mddefault}{\updefault}{\color[rgb]{0,0,0}$R'$}%
}}}}
\put(2971,-4066){\makebox(0,0)[lb]{\smash{{\SetFigFont{12}{14.4}{\rmdefault}{\mddefault}{\updefault}{\color[rgb]{0,0,0}$R''$}%
}}}}
\put(946,-4516){\makebox(0,0)[lb]{\smash{{\SetFigFont{12}{14.4}{\rmdefault}{\mddefault}{\updefault}{\color[rgb]{0,0,0}$T_1$}%
}}}}
\put(1576,-4471){\makebox(0,0)[lb]{\smash{{\SetFigFont{12}{14.4}{\rmdefault}{\mddefault}{\updefault}{\color[rgb]{0,0,0}$T_2$}%
}}}}
\put(991,-3886){\makebox(0,0)[lb]{\smash{{\SetFigFont{12}{14.4}{\rmdefault}{\mddefault}{\updefault}{\color[rgb]{0,0,0}$T_4$}%
}}}}
\put(1396,-4156){\makebox(0,0)[lb]{\smash{{\SetFigFont{12}{14.4}{\rmdefault}{\mddefault}{\updefault}{\color[rgb]{0,0,0}$T_3$}%
}}}}
\put(1216,-4921){\makebox(0,0)[lb]{\smash{{\SetFigFont{12}{14.4}{\rmdefault}{\mddefault}{\updefault}{\color[rgb]{0,0,0}$1/n$}%
}}}}
\put(316,-4516){\makebox(0,0)[lb]{\smash{{\SetFigFont{12}{14.4}{\rmdefault}{\mddefault}{\updefault}{\color[rgb]{0,0,0}$1/n^{5/2}$}%
}}}}
\put(181,-3931){\makebox(0,0)[lb]{\smash{{\SetFigFont{12}{14.4}{\rmdefault}{\mddefault}{\updefault}{\color[rgb]{0,0,0}$1/n^2$}%
}}}}
\end{picture}%
}
\caption{Tiling of the square $Q$.}
\label{F1}
\end{center}
\end{figure}
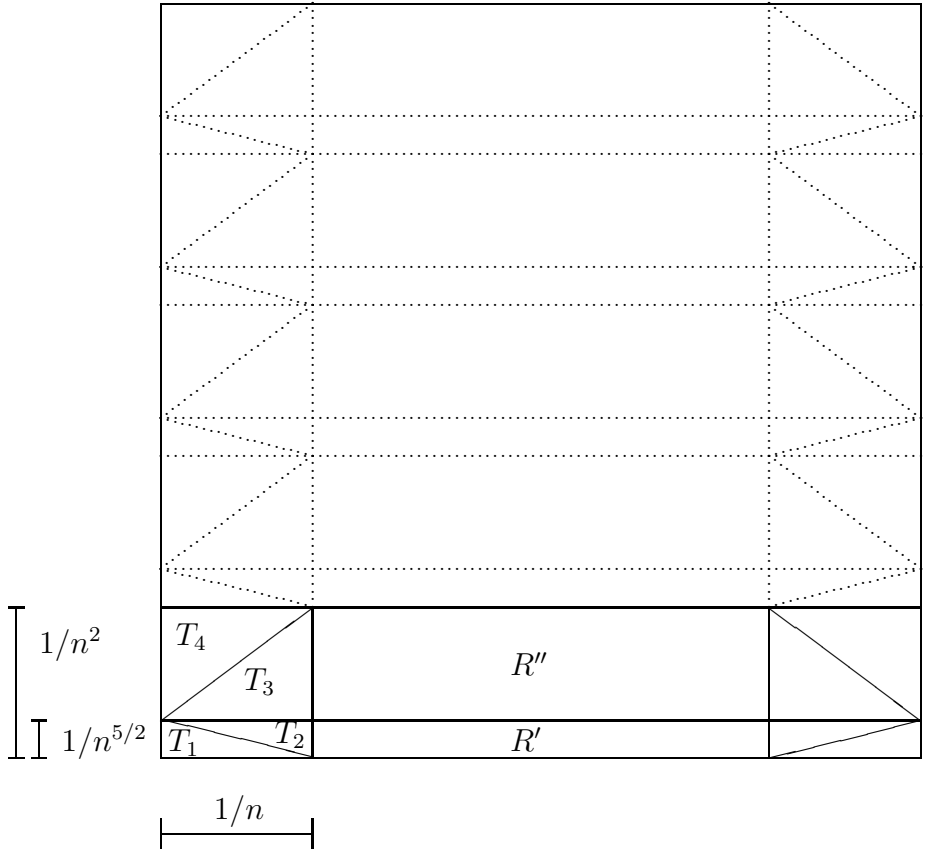

\begin{figure}[htbp]
\begin{center}
\scalebox{1}{
\includegraphics{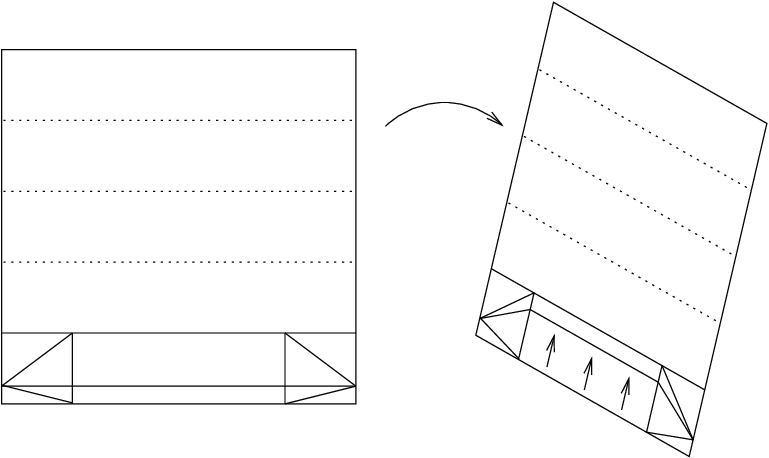}
}
\caption{Representation of the map $\phi_n$.}
\label{F2}
\end{center}
\end{figure}

Now we define the homeomorphism $\phi_n$ in the rectangle $R_0$, in such a way that 
\begin{equation}\label{defphin1}
{\phi_n}_{|\partial R_0}=\psi_{|\partial R_0}.
\end{equation}
Then we will be able to extend $\phi_n$ to the square $Q$, requiring its continuity and defining, for every $(x,y)\in {\rm{int}}(R_i)$,
\begin{equation}\label{defphin2}
\phi_n(x,y)=\phi_n((x,y)-(0, i/n^2))+i/n^2(b,d)
\end{equation}
(notice that the point $((x,y)-(0, i/n^2))$ belongs to ${\rm{int}}(R_0)$). We define $\phi_n$ on $R'$ as the linear map $\phi_n(x,y)=A'(x,y)^t$, where
$$A':=\left(\begin{array}{cc}
a & (n^{1/2}-1)b\\
c & (n^{1/2}-1)d\end{array}\right).$$
The map $\phi_n$ is now uniquely defined on $Q$ by the conditions \eqref{defphin1}, \eqref{defphin2} and by the requirement that $\phi_n$ is continuous on $Q$ and affine on $R',R''$ and on the triangles $T_1,\ldots,T_4$ (and on their symmetric copies).

In particular on $R''$ there holds
$$\nabla\phi_n=\left(\begin{array}{cc}
a & (1/(n^{1/2}-1))b\\
c & (1/(n^{1/2}-1))d\end{array}\right).$$

Denoting $F$ the set $$F:=\bigcup_{i=0}^{n^2-1}({\rm{int}}(R')+(0,i/n^2)),$$
it is easy to see that \eqref{eqF} holds. More precisely, since we want that $F$ is the union of pairwise disjoint open triangles, we should replace the set ${\rm{int}}(R')$ with the union of two disjoint open triangles such that the closure of their union is $R'$.
Notice also that $\phi_n=\psi$ on $T_1$ and $T_4$.

We now want to compute the variation $Var(\phi_n,Q)$.
Since $\phi_n$ is piecewise affine, this is equivalent to compute the \emph{energy}
$$\mathbb{E}(\phi_n):=\int_Q|\nabla\phi_n|_1 dx,$$
where we denoted by $|\cdot|_1$ the norm on ${\rm{Mat}}(2\times 2)$ given by 
$$\left|\left(\begin{array}{cc}
a & b\\
c & d\end{array}\right)\right|_1=|a|+|b|+|c|+|d|.$$
By construction, we have
$$\mathbb{E}(\phi_n)=n^2\mathbb{E}({\phi_n}_{|R_0}).$$
Moreover, by the 1-homogeneity of the energy, we can compute 
\begin{equation}\label{eq1}
\begin{split}
n^2\mathbb{E}({\phi_n}_{|R'})&=
n^2\left(1-\frac{2}{n}\right)\left(\frac{1}{n^{5/2}}\right)\left|\left(\begin{array}{cc}
a & (n^{1/2}-1)b\\
c & (n^{1/2}-1)d\end{array}\right)\right|_1\\
&=\left(1-\frac{2}{n}\right)\left|\left(\begin{array}{cc}
(1/n^{1/2})a & (1-(1/n^{1/2}))b\\
(1/n^{1/2})c & (1-(1/n^{1/2}))d\end{array}\right)\right|_1
\end{split}
\end{equation}
and analogously
\begin{equation}\label{eq2}
\begin{split}
n^2\mathbb{E}({\phi_n}_{|R''})&=
n^2\left(1-\frac{2}{n}\right)\left(\frac {1}{n^2}-\frac{1}{n^{5/2}}\right)\left|\left(\begin{array}{cc}
a & (1/(n^{1/2}-1))b\\
c & (1/(n^{1/2}-1))d\end{array}\right)\right|_1\\
&=\left(1-\frac{2}{n}\right)\left|\left(\begin{array}{cc}
(1-(1/n^{1/2}))a & (1/n^{1/2})b\\
(1-(1/n^{1/2}))c & (1/n^{1/2})d\end{array}\right)\right|_1.
\end{split}
\end{equation}
Combining \eqref{eq1} and \eqref{eq2} we have that
\begin{equation}\label{eq3}
\mathbb{E}({\phi_n}_{|(R'\cup R'')})=\mathbb{E}(\psi_{|(R'\cup R'')}).
\end{equation}

Regarding the energy of $\phi_n$ in the triangles $T_1,\ldots, T_4$, we have, trivially
$$\mathbb{E}({\phi_n}_{|T_1})=\mathbb{E}({\psi}_{|T_1})$$
and 
$$\mathbb{E}({\phi_n}_{|T_4})=\mathbb{E}({\psi}_{|T_4}).$$
Moreover on $T_2$, we have
$$\nabla\phi_n=\left(\begin{array}{cc}
a+((1/n)-2/n^{3/2})b & (n^{1/2}-1)b\\
c+((1/n)-2/n^{3/2})d & (n^{1/2}-1)d\end{array}\right)$$
and therefore 
\begin{equation}\label{eq4}
{|\nabla\phi_n|}_1\leq n^{1/2}|\nabla\psi|_1.
\end{equation}
Finally on $T_3$, we have
$$\nabla\phi_n=\left(\begin{array}{cc}
a+((1/n)-2/n^{3/2})b & (1/(n^{1/2}-1))b\\
c+((1/n)-2/n^{3/2})d & (1/(n^{1/2}-1))d\end{array}\right)$$
and therefore 
\begin{equation}\label{eq5}
{|\nabla\phi_n|}_1\leq 2|\nabla\psi|_1.
\end{equation}

With similar computations, one can verify that the equations \eqref{eq4} and \eqref{eq5} are satisfied also on the symmetric copies of $T_2$ and $T_3$, respectively. Hence, combining \eqref{eq3} with \eqref{eq4} and \eqref{eq5}, we have
$$\mathbb{E}({\phi_n})\leq(1+(2/n^{1/2}))\mathbb{E}(\psi).$$
\end{proof}

\section{Proof of Theorem \ref{main}}
\label{s4}

We have to prove that the sets $A_n$ defined in the Introduction are open and dense in $(X,d_X)$.

\begin{lemma}\label{lopen}
For every $n\in\N$ the set $A_n$ is open.
\end{lemma}
\begin{proof}
Take $f\in A_n$. Let $1/n>\varepsilon>0$ and let $T_1,\ldots, T_m$ be pairwise disjoint open triangles in $Q$ such that, denoting $E=\bigcup_i T_i$, there holds
\begin{enumerate}
\item $|E|<1/n-\varepsilon$;
\item $|f(E)|>1-1/n+\varepsilon$.
\end{enumerate}
Since the image of each triangle $T_i$ is open, then there exists $\eta>0$ such that, denoting for every open set $B$
$$B^\eta:= \{x\in B: dist(x,B^C)>\eta\},$$
there holds
$$|f(T_i)^\eta|\geq(1-\varepsilon)|f(T_i)|,$$
for every $i=1,\ldots,m$.

Consider now $g\in X$ with $d_X(f,g)<\eta$, In particular $\|f-g\|_\infty<\eta$, hence $g(T_i)\supset f(T_i)^\eta$, for every $i=1,\ldots,m$. Therefore 
$$|g(E)|=\sum_{i=1}^m|g(T_i)|\geq\sum_{i=1}^m|f(T_i)^\eta|\geq(1-\varepsilon)|f(E)|>1-1/n.$$
Hence $g\in A_n$.
\end{proof}

\begin{lemma}\label{ldense}
For every $n\in\N$ the set $A_n$ is dense in $X$.
\end{lemma}
\begin{proof}
Fix $f\in X$ and $\varepsilon>0$. 
We want to find $f_{\varepsilon}\in A_n$ with ${d_X(f,f_{\varepsilon})<\varepsilon}$. By \cite[Theorem A]{PR}, we can find a sequence of (finitely) piecewise affine homeomorphisms $(g_i)_{i\in\N}:Q\to Q$ such that
\begin{enumerate}
\item ${g_i}_{|\partial Q}=Id;$
\item $\|g_i-f\|_\infty$ tends to 0 for $i\to\infty$;
\item $\|g_i^{-1}-f^{-1}\|_\infty$ tends to 0 for $i\to\infty$;
\item $\lim_{i\to\infty}Var(g_i,Q)\leq Var(f,Q)$.
\end{enumerate}

Given (2), the validity of (3) is a simple consequence of the uniform continuity of $f^{-1}$. Indeed such property implies that if $\|g_i-f\|_\infty$ is small, then $\|g_i^{-1}-f^{-1}\|_\infty$ is also small. To prove it, fix $\varepsilon>0$ and let $\delta>0$
be such that if $|x-y|<\delta$ then $|f^{-1}(x)-f^{-1}(y)|<\varepsilon$. Now take $i\in\N$ such that $\|g_i-f\|_\infty<\delta$. We want to prove that $\|g_i^{-1}-f^{-1}\|_\infty<\varepsilon$. Assume by contradiction there exists $x_0$ such that
$|g_i^{-1}(x_0)-f^{-1}(x_0)|>\varepsilon$. Denoting $x_1:=g_i^{-1}(x_0)$ and $x_2:=f^{-1}(x_0)$, we have $|g_i(x_1)-f(x_1)|<\delta$. Hence, denoting $x_3:=f(x_1)$, we have $|x_3-x_0|<\delta$, but $|f^{-1}(x_3)-f^{-1}(x_0)|>\varepsilon$, which is a contradiction.\\

Using the lower semicontinuity of the variation w.r.t. the uniform convergence, from (1)-(4) we deduce that there exists a piecewise affine homeomorphism $g_{\varepsilon}\in X$ with
\begin{equation}\label{eq6}
d_X(f,g_{\varepsilon})<\varepsilon/4.
\end{equation}
We can also assume
\begin{equation}\label{eq7}
\left|\frac{1}{M-Var(f,Q)}-\frac{1}{M-Var(g_{\varepsilon},Q)}\right|<\varepsilon/8.
\end{equation}

Now we take a finite triangulation of $Q$ such that $g_{\varepsilon}$ is affine on each triangle. If necessary, we can refine such triangulation in order to obtain a new finite triangulation $\tau$ such that the diameter of all triangles $T_i\in\tau$ and of their images through $g_{\varepsilon}$ are less than $\varepsilon/8$.

By \eqref{eq7} we can take $m\in\N$, $m>2n$ such that 
\begin{equation}\label{eq8}
\left|\frac{1}{M-Var(f,Q)}-\frac{1}{M-CVar(g_{\varepsilon},Q)}\right|<\varepsilon/2,
\end{equation}
for every $C\in [1-(1/m),1+1/m]$.
We define the homeomorphism $f_{\varepsilon}$ as follows.
For every $T_i\in\tau$ take finitely many closed squares $Q_i^j$ with pairwise disjoint interiors and with edges parallel to the coordinate lines such that $Q_i^j\subset T_i$ and
\begin{equation}\label{eq9}
\left|\bigcup_jQ_i^j\right|\geq \left(1-\frac{1}{m}\right)|T_i|.
\end{equation}

For every $i,j$, define $f_{\varepsilon}$ on $Q_i^j$ as the map obtained by replacing the map $g_{\varepsilon}$ with the map ${(g_{\varepsilon})}_m$ obtained applying Proposition \ref{propphin} with $n=m$ and $\phi={g_{\varepsilon}}_{|Q_i^j}$.
For every $i$, define $f_{\varepsilon}:=g_{\varepsilon}$ on $T_i\setminus(\bigcup_jQ_i^j)$.

By \eqref{eq6},\eqref{eq8}, point (1) of Proposition \ref{propphin} and the property of the triangulation, we have
$$d(f,f_{\varepsilon})\leq d(f,g_{\varepsilon})+d(f_{\varepsilon},g_{\varepsilon})<((\varepsilon/4)+(\varepsilon/2))+(\varepsilon/4)=\varepsilon.$$

By point (3) of Proposition \ref{propphin} and \eqref{eq9} we have that, denoting by $(F_i^j)_m$ the set given by Proposition \ref{propphin} applied with $n=m$ and $\phi={g_{\varepsilon}}_{|Q_i^j}$, and by $F$ the set
$$F:=\bigcup_{i,j}(F_i^j)_m,$$
there holds $|F|< 1/m$ and 
$$|f_{\varepsilon}(F)|>(1-(1/m))(1-(1/m))>1-(2/m)>1-(1/n),$$
hence $f_{\varepsilon}\in A_n$.
\end{proof}

\begin{proof}[Proof of Theorem \ref{main}]
The only thing left to show is that $A\supset\bigcap_{n\in\N} A_n$. Fix $f\in\bigcap_{n\in\N} A_n$. In particular, for every $j\in\N$, we have $f\in\bigcap_{i>j} A_{2^i}$,
hence for every $i\in\N$ with $i>j$ there exists a Borel set $E_i$ with $|E_i|<2^{-i}$ such that $|f(E_i)|>1-2^{-i}$. Therefore denoting $E^j:=\bigcup_{i>j}E_i$, we have $|E^j|<2^{-j}$ and $|f(E^j)|=1$. Since the countable intersection of sets of full measure is a set of full measure  we deduce that,
denoting $E:=\bigcap_{j\in\N}E^j$, we have that $|E|=0$ and $|f(E)|=1$, hence $f\in A$.
\end{proof}

\section{Final remarks}
\label{s5}
\subsection{Higher dimension}
We believe that a natural generalization of Theorem \ref{main} holds also in dimension $d\geq 3$. If the approximation result of \cite{PR} was proved to be valid in any dimension, such generalization could be proved exactly with the argument used above. Namely, the basic building block constructed in Section \ref{s3} could be extended in an affine way along the additional directions and this would preserve the necessary estimates. At the moment, the problem of extending the approximation result of \cite{PR} to higher dimensions is widely open. Nevertheless we also believe that the generalization of our result could be proved independently from the validity of the approximation result, roughly via a composition of our basic building block (or more precisely of the affine extension mentioned above) with the original homeomorphism, locally around points of approximate continuity of the absolutely continuous part of the distributional gradient. For the sake of simplicity, we do not pursue this in the present paper.
\subsection{$W^{1,p}$ homeomorphisms}
In \cite{He}, Hencl proves that for $1\leq p <2$ there exists a homeomorphism $f:Q\to Q$ in $W^{1,p}$ with $f_{|\partial Q}=Id$ satisfying $Jf=0$ a.e. It is easy to see that the set of such homemomorphisms is dense in $W^{1,p}$ with respect to the $C^0$-distance. Therefore a natural question is whether in a suitable complete metric space of $W^{1,p}$ homeomorphisms, these maps are residually many. For $p>1$ a natural setting to answer this question (i.e. a reasonable complete metric space of $W^{1,p}$ homeomorphisms) is the set
$$Y:=\left\{f:Q\rightarrow Q:\; f\;{\rm{is\;a\;}}W^{1,p}\;{\rm{homeomorphism}},\; f_{|\partial Q}=Id,\; \int_Q|Df|^p\leq M\right\}$$
for an arbitrary constant $M>1$, with the distance
$$d_Y(f,g):=\|f-g\|_\infty+\|f^{-1}-g^{-1}\|_\infty.$$
Observe that one cannot consider as distance the natural norm of $W^{1,p}$, because the convergence in such norm would also imply the convergence of the Jacobians, almost everywhere.

Trying to imitate our strategy, one could rely on the fact that in \cite{IKO} the authors prove that it is possible to approximate a $W^{1,p}$ homeomorphism $(p>1)$ uniformly and in the $W^{1,p}$ norm by piecewise affine homeomorphisms. Nevertheless, there is no hope that homeomorphisms with zero Jacobian almost everywhere are residual in the metric space $(Y,d_Y)$, since they are not even dense. Indeed, take a homeomorphism $f\in Y$ with $Jf>0$ on a set of positive measure and satisfying $\int_Q|Df|^p=M$. Assume that there exist homeomorphisms $f_n$ in $Y$ with $d_Y(f_n,f)\to 0$ as $n\to\infty$. Since the quantity $\int_Q|Df|^p$ is lower semicontinuous with respect to the uniform convergence, this would force 
\begin{equation}\label{efin}
\int_Q|Df_n|^p\rightarrow M=\int_Q|Df|^p, \quad \mbox{as $n\to\infty$.}
\end{equation}
In turn, since the norm on $W^{1,p}$ is uniformly convex, the uniform convergence and \eqref{efin} imply the convergence in norm, which forces the convergence of the Jacobians, too. Indeed, in every uniformly convex space, if $x_n\rightharpoonup x$ and $\|x_n\|\to\|x\|$ then $\|x_n-x\|\to 0$ (see Proposition 3.32 of \cite{Br}). In particular we can deduce that it is not possible to extend Proposition \ref{propphin} to the setting of $W^{1,p}$ homeomorphisms: roughly speaking, in this class there is a positive minimal cost in the energy to approximate an affine homeomorphism with homeomorphisms that map a small set onto a large one. Notice that, since the subset of homeomorphisms $f\in Y$ satisfying $\int_Q|Df|^p=M$ is residual in $Y$, then we have actually proved that the set of $W^{1,p}$ homeomorphisms with zero Jacobian almost everywhere is of first category in $Y$. 
Indeed we have proved that the set of homeomorphisms $f\in Y$ satisfying $\int_Q|Df|^p=M$ and $Jf>0$ on a set of positive measure is relatively open. To prove that it is dense, we can use the same construction described at the end of \S \ref{s2}. Moreover the result is independent on the dimension of the ambient space: let us summarize all these observations in the following
\begin{theorem}\label{firstcatp}
Let $d\geq 2$ and $Q^d:=(0,1)^d$. Fix $1<p<d$, $M>1$. Define
$$Y:=\left\{f:Q^n\rightarrow Q^d: f\;{\rm{is\;a\;}}W^{1,p}\;{\rm{homeomorphism}},\; f_{|\partial Q^d}=Id, \int_{Q^d}|Df|^p\leq M\right\}$$
and the distance on $Y$
$$d_Y(f,g):=\|f-g\|_\infty+\|f^{-1}-g^{-1}\|_\infty.$$
Then the set $A$ of all homeomorphisms $f\in Y$ with $Jf=0$ a.e. is of first category in $Y$, i.e. $Y\setminus A$ is residual in $Y$.
\end{theorem}

\subsection{$W^{1,1}$ homeomorphisms}
In \cite{HP} the authors prove that it is possible to approximate a $W^{1,1}$ homeomorphism uniformly and in the $W^{1,1}$ norm by piecewise affine homeomorphisms. Moreover in Proposition \ref{propphin} (1), it is equivalent to consider the variation $Var(\phi,Q)$ or the energy $\mathbb{E}(\phi)$, hence if one considers the metric space $(Y,d_Y)$ as defined in the previous subsection, for $p=1$, it is not difficult to adapt the arguments presented in Section \ref{s4}, to prove that the set of $W^{1,1}$ homeomorphisms of $Q$ onto itself mapping a set of measure smaller than $1/n$ onto a set of measure larger than $1-1/n$ are open and dense. The issue here is that $(Y,d_Y)$ is not complete, because a sequence of $W^{1,1}$ maps converging uniformly and with equi-bounded energies may converge to a map which is in $BV$ but not in $W^{1,1}$. Hence the countable intersection of open and dense sets might in principle be empty. The completion of such space is a space of $BV$ homeomorphisms, with a uniform bound on the variation. However, such complete metric space is too large for $W^{1,1}$ homeomorphisms with zero Jacobian almost everywhere to be residual. Indeed in Proposition \ref{firstcat1}, we show that in such metric space the subset of $W^{1,1}$ homeomorphisms is of first category. For the sake of brevity, and since the result is not surprising, we prove such statement only in the metric space defined in the Introduction. The same can be done, with minor changes, in the completion of the metric space defined in the previous subsection, for $p=1$. 

\begin{proposition}\label{firstcat1}
Let $(X,d_X)$ be the metric space defined in the Introduction. Then the set $A$ of all $W^{1,1}$ homeomorphisms in $X$ is of first category. 
\end{proposition}
\begin{proof}
Define
$$A_n:=\{f\in X: \exists E\subset Q, |E|<1/n, Var(f,E)>1/2-1/n\},$$
where $E$ is the union of finitely many pairwise disjoint open triangles.
Clearly the intersection of the $A_n$'s does not contain any $W^{1,1}$ homeomorphism, therefore, to prove the proposition it is sufficient to show that the $A_n$'s are open and dense. The openness is just a consequence of the lower semicontinuity of the variation with respect to the uniform convergence. The density can be achieved as in Lemma \ref{ldense}: it is sufficient to observe, from \eqref{eq1} that the maps $\phi_n, \psi$ and the set $F$ constructed in \S \ref{s3} satisfy 
\begin{equation}\label{efin2}
Var(\phi_n, F)>(1/2-1/\sqrt{n})Var(\psi, Q).
\end{equation}
More precisely, such inequality is satisfied tout court if $|b|+|d|\geq|a|+|c|$. In case $|b|+|d|<|a|+|c|$, actually one should slightly modify the map $\phi_n$ to obtain \eqref{efin2}: roughly speaking, it is sufficient to ``switch'' the coordinates $(x,y)$.
\end{proof}

\bibliographystyle{plain}

%
%

\vskip .5 cm

{\parindent = 0 pt\begin{footnotesize}

Andrea Marchese
\\
Institut f\"ur Mathematik,
Mathematisch-naturwissenschaftliche Fakult\"at,
Universit\"at Z\"urich\\
Winterthurerstrasse 190,
CH-8057 Z\"urich,
Switzerland\\
e-mail: {\tt andrea.marchese@math.uzh.ch}

\end{footnotesize}
}

\end{document}